\documentclass[12pt]{article}
\usepackage{a4}
\usepackage{latexsym}
\usepackage{amsfonts}
\usepackage{amssymb}
\usepackage{amsmath}
\usepackage{amsthm}
\newtheorem{theorem}{Theorem}

\newtheorem{lemma}[theorem]{Lemma}

\newtheorem*{interpretation}{Interpretation}
\title{An improved incidence bound for fields of prime order}
\author{Timothy G. F. Jones\footnote{School of Mathematics, University of Bristol BS8 1TW, United Kingdom, tgf.jones@bristol.ac.uk}}
\date{}

\begin{document}

\maketitle
\begin{abstract}
Let $P$ be a set of points and $L$ a set of lines in $\mathbb{F}_p^2$, with $|P|,|L|\leq N$ and $N<p$. We show that $P$ and $L$ generate no more than $C N^{\frac{3}{2}-\frac{1}{806}+o(1)}$ incidences for some absolute constant $C$. This improves by an order of magnitude on the previously best-known bound of $C N^{\frac{3}{2}-\frac{1}{10,678}}$.
\end{abstract}

\section*{Notation}
\subsection*{Asymptotic notation}
If $f$ and $g$ are functions and there is a constant $C$ such that $f(x)\leq C g(x)$ for all $x$ then we will write $f\ll g$, $f=O(g)$, or $g=\Omega(f)$. 

If there is a constant $c$ such that $f \ll g \log^c x$ then we will write $f \lesssim g$, $f=\widetilde{O}(g)$ or $g=\widetilde{\Omega}(f)$.

If $f(x) \lesssim g(x)$ and $g(x) \lesssim f(x)$ then we will write $f(x) \approx g(x)$. 

The implicit constants $C$ and $c$ may vary with different instances of the notation. If one or other of them depends on a parameter $\epsilon$ then this will be indicated by a subscript, e.g. $f=O_{\epsilon}(g)$ etc.
\subsection*{Sumset notation}
If $A$ and $B$ are subsets of a field $F$, then we will write 
$$A+B=\left\{a+b:a \in A, b \in B\right\}.$$
If $G \subseteq A \times B$ then we will write
$$A\stackrel{G}{+}B=\left\{a+b:(a,b)\in G\right\}.$$
These definitions extend analogously to subtraction, multiplication and division.

\section{Introduction}
\subsection{Incidences}
This paper is about counting incidences between points and lines in a plane. A point is \textit{incident} to a line if it lies on that line. Incidences are counted with multiplicity, in the sense that several lines incident to the same point determine several incidences, and vice versa.

We are interested in knowing the maximum number of incidences between a set $P$ of points and a set $L$ of lines, say with $|P|,|L|\leq N$. Certainly this cannot exceed $N^2$. But using the Cauchy-Schwartz inequality and the fact that two distinct points determine a line, it is straightforward to see that it is in fact $O(N^{3/2})$. 

So, writing $I(P,L)$ for the number of incidences between $P$ and $L$, \textit{non-trivial} incidence bounds are of the form $I(P,L)\ll N^{3/2-\epsilon}$ with $\epsilon>0$. The larger the value of $\epsilon$, the stronger the bound. 

How large can $\epsilon$ be? A natural example shows that the best that can be hoped for is $\epsilon = 1/6$. Progress towards achieving this depends on the ambient field over which points and lines are defined. Over $\mathbb{R}$ (i.e. when the points and lines lie in $\mathbb{R}^2$) the best possible result was obtained by Szemer\'edi and Trotter \cite{ST}:

\begin{theorem}[Szemer\'edi-Trotter]\label{theorem:ST}
Let $P$ be a set of points and $L$ a set of lines in $\mathbb{R}^2$ with $|P|,|L|\leq N$. Then $I(P,L)\ll N^{4/3}=N^{\frac{3}{2}-\frac{1}{6}}$.
\end{theorem}

This result was generalised to $\mathbb{C}$ by Toth \cite{toth}, and a near-sharp generalisation to higher dimensional points and varieties was recently given by Solymosi and Tao \cite{solymositao}.

When working over finite fields, the known bounds are considerably weaker. Helfgott and Rudnev \cite{HR} obtained:

\begin{theorem}[Helfgott-Rudnev] \label{theorem:HR}
Let $\mathbb{F}_p$ be the finite field of prime order $p$. Let $P$ be a set of points and $L$ a set of lines in $\mathbb{F}_p^2$ with $|P|,|L|\leq N$ and $N<p$. Then $I(P,L)\ll N^{\frac{3}{2}-\frac{1}{10,678}}$. 
\end{theorem}

This followed work of Bourgain, Katz and Tao \cite{BKT}, which established the existence of a non-zero $\epsilon$ so long as $N<p^{2-\delta(\epsilon)}$, but did not quantify it.

The condition $N<p$ in Theorem \ref{theorem:HR} is essentially a nondegeneracy requirement; it prevents $P$ from being the entire plane $\mathbb{F}_p^2$, in which case no nontrivial bound would be possible. The present author \cite{me1} extended Theorem \ref{theorem:HR} to a finite field $\mathbb{F}_q$ of general order, subject to analogous conditions to prevent $P$ from being a large part of a subplane, and with a slightly weaker exponent. Vinh \cite{vinh} also gave a nontrivial incidence bound for sets that are a large part of $\mathbb{F}_q^2$, but still not the entire plane.

\subsection{Relation to sum-product estimates}

Erd\"os and Szemer\'edi conjectured that any finite set $A \subseteq \mathbb{R}$ must satisfy

$$\max\left\{|A+A|,|A\cdot A|\right\}\gg_{\epsilon} |A|^{2-\epsilon}$$

for any $\epsilon>0$. Since sets with small sumset are in some sense additively structured, and similarly for product sets, this asserts that no set can be simultaneously additively and multiplicatively structured. Bounds on $\max\left\{|A+A|,|A\cdot A|\right\}$ are called \textit{sum-product estimates}.

As with incidences, the sum-product question can be just as easily posed over other fields, and the known results depend on the choice of field. Over $\mathbb{R}$ the best-known result is due to Solymosi \cite{solymosi1}, who obtained 
$\max\left\{|A+A|,|A\cdot A|\right\}\gtrsim |A|^{4/3}.$
The best known bound over $\mathbb{F}_p$, due to Rudnev \cite{rudnev}, is $|A|^{12/11}$ so long as $|A|<\sqrt{p}$.

Incidence and sum-product estimates are closely linked. For example the Szemer\'edi-Trotter theorem was used to establish the previously best-known sum-product result over $\mathbb{R}$ of $\max\left\{|A+A|,|A\cdot A|\right\}\gtrsim |A|^{14/11}.$ This was again due to Solymosi \cite{solymosi2} and built on previous work of Elekes \cite{elekes}. In the other direction, the Helfgott-Rudnev and Bourgain-Katz-Tao incidence bounds over $\mathbb{F}_p$ were established using finite field sum-product results and methods.

\subsection{Results}
This paper proves the following theorem, which improves on Theorem \ref{theorem:HR} by an order of magnitude:

\begin{theorem}\label{theorem:result}
Let $\mathbb{F}_p$ be the finite field of prime order $p$. Let $P$ be a set of points and $L$ a set of lines in $\mathbb{F}_p^2$ with $|P|,|L|\leq N$ and $N<p$. Then $I(P,L)\ll N^{\frac{3}{2}-\frac{1}{806}+o(1)}$. 
\end{theorem}

As will be seen, the proof of Theorem \ref{theorem:result} uses finite field sum-product estimates as a `black box'. It is likely that further improvements can be obtained by unpacking the sum-product proof and exploiting its use of multiplicative energy and covering arguments to bypass some costly Balog-Szemer\'edi-Gowers type refinements. 

It should also be remarked that by combining recent sum-product work of Li and Roche-Newton \cite{LRN} with the approach in \cite{me1}, the exponent in Theorem \ref{theorem:result} extends to a general finite field $\mathbb{F}_q$, subject to appropriate non-degeneracy conditions. This is omitted in the interests of simplicity and instead left as an exercise for the motivated reader.

\subsection*{A sketch of the proof}

The proof of Theorem \ref{theorem:result} relies on sum-product estimates over $\mathbb{F}_p$, the most recent of which is that of Rudnev \cite{rudnev}. As Rudnev notes in his paper, the result can be expressed as a `difference-ratio' estimate as well. We will use the following formulation:

\begin{lemma}[Rudnev]\label{theorem:rudnev}
Let $A \subseteq \mathbb{F}_p$ with $|A|<\sqrt{p}$. Then $\max\left\{|A-A|,|A/ A|\right\}\gg |A|^{12/11}$.
\end{lemma}

Through Balog-Szemer\'edi-Gowers type refinements, this can be interpreted as follows:

\begin{interpretation}
A point-set in $\mathbb{F}_p^2$ that is efficiently contained in a direct product $A \times B$ with $|A|\approx |B|<\sqrt{p}$  cannot be simultaneously covered by only a few lines through a fixed point and only a few lines with a fixed (non-vertical or horizontal) gradient. Here `few' means less than $|A|^{1+\delta}$ for a small $\delta>0$. 
\end{interpretation}

The proof of Theorem \ref{theorem:result} begins by supposing for a contradiction that $I(P,L)$ is large. By averaging arguments we can find four points $p_1,p_2,p_3,p_4 \in P$ and a large subset $R$ of $P$ that is supported on only a few lines through each of the four $p_i$. Moreover, the points $p_1,p_2,p_3$ are colinear, and the point $p_4$ lies off their common line.

Apply a projective transformation $\tau$ so that lines through $\tau(p_1)$ and $\tau(p_2)$ are parallel to the vertical and horizontal co-ordinate axes respectively. This means that the image $\tau(R)$ is efficiently contained in a direct product $A \times B$. 

Since $\tau$ sends $p_1$ and $p_2$ to the line at infinity, and $p_3$ is colinear with both, $\tau$ must also send $p_3$ to the line at infinity. All lines passing through $\tau(p_3)$ are then parallel, but not parallel to either co-ordinate axis. So $\tau(R)$ is covered by only a few lines, all of some fixed gradient. 

Finally, since $p_4$ is not colinear with the other three points, $\tau(p_4)$ is not on the line at infinity. So $\tau(R)$ is covered by only a few lines, all through some fixed point.

But $\tau(R)$ then satisfies the conditions that were explicitly banned in the interpretation of Lemma \ref{theorem:rudnev}, and so we obtain a contradiction.

\subsection*{Acknowledgements}
The author is grateful to T. Bloom, A. Glibichuk and M. Rudnev for reading an earlier draft, and to O. Roche-Newton and M. Rudnev for conversations.

\section{Proof of Theorem \ref{theorem:result}}

Suppose that $I(P,L)\gg N^{3/2-\epsilon}$. We shall show that $\epsilon\geq \frac{1}{806}-o(1)$.

Without loss of generality we may assume that every point in $P$ is incident to only $O(N^{1/2+\epsilon})$ lines in $L$ and every line in $L$ is incident to only $O(N^{1/2+\epsilon})$ points in $P$. Indeed, if $P_+$ is the set of points incident to at least $C N^{1/2+\epsilon}$ lines then, writing $\delta_{pl}=1$ if $p \in l$ and $0$ otherwise, we have
$$I(P_+,L)=\sum_{p \in P_+}\sum_{l \in L}\delta_{pl}\leq \frac{1}{CN^{1/2+\epsilon}}\sum_{p \in P_+}\sum_{l_1,l_2 \in L}\delta_{pl_1}\delta_{pl_2} \leq \frac{2}{C} N^{3/2+\epsilon}$$
and so $P_+$ can be discarded without affecting incidences by more than a constant. The argument for lines is similar.

\subsection{Partial sum-products}
Our first aim is to show that if $\epsilon$ is small then there must exist large $A,B \subseteq \mathbb{F}_p$ and a large $E \subseteq A \times B$ such that $A\stackrel{E}{-}B$ and $A\stackrel{E}{/}B$ are both small.
This is accomplished by making two iterations of a refinement argument, and then applying a projective transformation.

\subsubsection{First refinement}

Let $P_1$  be the set of points in $P$ incident to $\Omega\left(N^{1/2-\epsilon}\right)$ lines in $L$.
By a dyadic pigeonholing we may find a subset $P_1'\subseteq P_1$ and an integer $K$ with $$N^{1/2- \epsilon}\ll K \ll N^{1/2+\epsilon}$$ such that every point in $P_1'$ is incident to between $K$ and $2K$ lines in $L$, and $I(P_1',L)\approx|P_1'|K \approx N^{3/2-\epsilon}$. 

Now let $L_1$ be the set of lines in $L$ incident to $\widetilde{\Omega}\left(N^{1/2-\epsilon}\right)$ points in $P_1'$, and
 $P_2$ be the set of points in $P_1'$ incident to $\widetilde{\Omega}\left(K\right)$ lines in $L_1$.

Note that $I(P_2,L_1)\approx I(P_1',L_1)\approx I(P_1',L)\approx N^{3/2-\epsilon}$ because in each case we can pick constants to ensure that only a fraction of the $N^{3/2-\epsilon}$ incidences are discarded. Note also that $|P_2|\gtrsim N^{1-2\epsilon}$ and in particular that $P_2$ has at least say 100 elements.

For each $p \in P_2$ let $P_p=\left\{q \in P_1':l_{pq}\in L_1\right\}$. Since $p$ is incident to $\widetilde{\Omega}(K)$ lines in $L_1$ and each line is incident to at least $\widetilde{\Omega}(N^{1/2-\epsilon})$ points in $P_1'$ we have $|P_p|\gtrsim K N^{1/2-\epsilon}$. Summing over all such $p$ and applying Cauchy-Schwartz gives
\begin{equation}\label{eq:CS}
K N^{1/2-\epsilon}|P_2|\lesssim \sum_{p \in P_2}|P_2|
\leq |P_1|^{1/2}\left(\sum_{p_1,p_2}\left|P_{p_1}\cap P_{p_2}\right|\right)^{1/2}
\end{equation}
and so 
$$\frac{K^2 |P_2|^2}{N^{2\epsilon}}  \lesssim \sum_{p_1,p_2 \in P_2}\left|P_{p_1}\cap P_{p_2}\right|$$
Now, $\sum_{p \in P_2}|P_p| \leq |P_2|N$, and this is certainly far less that $\frac{K^2 |P_2|^2}{N^{2\epsilon}}$ or we are already done. So we may assume that the summation on the right is over distinct pairs $p_1,p_2$. There then exist distinct $p_1,p_2$ with $\left|P_{p_1}\cap P_{p_2}\right|\gtrsim K^2/N^{2 \epsilon}$. 

\subsubsection{Second refinement}

Define $Q=P_{p_1}\cap P_{p_2}$, so that (passing to a subset if necessary) $$|Q|\approx K^2/N^{2\epsilon}$$ and $$I(Q,L)\approx K^3/N^{2\epsilon}.$$ Moreover, note that $Q$ is supported over only $O(K)$ lines through each of $p_1$ and $p_2$.

We now build a similar argument for the set $Q$. We let $J$ be the set of lines in $L$ incident to $\widetilde{\Omega}\left(K^3/N^{1+2\epsilon}\right)$ points in $Q$, and $Q_1$ be the set of points in $Q$ incident to $\widetilde{\Omega}\left(K\right)$ lines in $J$. Note that $I(Q_1,J)\approx I(Q,J)\approx I(Q,L)\approx K^3/N^{2\epsilon}$.

Points in $Q_1$ are supported over $O(K)$ lines through $p_1$, each incident to $O(N^{1/2+\epsilon})$ points, and we also know $|Q_1|\gtrsim N^{1-6 \epsilon}$. So there are several (in fact $\widetilde{\Omega}(N^{1/2-7\epsilon})$) lines in $J$ that are each incident to $p_1$ and also $\widetilde{\Omega}(N^{1/2-7\epsilon})$ points in $Q$. Pick one such line that is not incident to $p_2$. Let $Q_2$ be the set of points from $Q_1$ on that line.

For each $p \in Q_2$ let $Q_p=\left\{q \in Q:l_{pq}\in J\right\}$, so that $|Q_p|\gtrsim K^4/N^{1+2 \epsilon}$. We have 

$$\frac{K^4}{N^{1+2\epsilon}}|Q_2|\lesssim \sum_{p \in Q_2}|Q_p|$$

and so applying Cauchy Schwartz as in (\ref{eq:CS}) we find distinct $p_3,p_4 \in Q_2$, not equal to $p_1,p_2$, such that $|Q_{p_3}\cap Q_{p_4}|\gtrsim \frac{K^8}{N^{2+4\epsilon}|Q|}\approx  \frac{K^6}{N^{2+2\epsilon}}$. Define $R= Q_{p_3}\cap Q_{p_4}$ so that $$|R|\gtrsim K^6/N^{2+2 \epsilon}$$ and $R$ is supported on $O(K)$ lines through each of $p_1,p_2,p_3,p_4$.

\subsubsection{Projective transformation}

Since $p_3,p_4 \in Q_2$ they are colinear with $p_1$ but not $p_2$. Pick a projective transformation $\tau$ that sends their common line to the line at infinity, in such a way that lines through $p_3$ and $p_4$ are parallel to the horiztonal and vertical co-ordinate axes respectively. If we define $E=\tau(R)$ then this means that $E \subseteq A \times B$ for $A,B \subseteq \mathbb{F}_p$. By appending points if neccesary we may assume $$|A|,|B| \approx K.$$ And by discarding points we may assume that $$|E|\approx K^6/N^{2+2 \epsilon}.$$

Since $p_1$ is colinear with $p_3$ and $p_4$, lines through $\tau(p_1)$ are all parallel, but not parallel to the vertical or horizontal axes. So there exists $z \notin \left\{0,\infty\right\}$ such that $E$ is supported on only $O(K)$ lines of gradient $z$. 

Since $p_2$ is not colinear with $p_3$, the map $\tau$ sends it to the affine plane, and so $E$ is supported on $O(K)$ lines all passing through $\tau(p_2)$.

Note that these properties are invariant under translation of the plane and scaling of a single co-ordinate axis in the following sense. First, if $T$ is some translation of the plane then $T(E)$ is supported on $O(K)$ lines through $T(\tau(p_2))$, and still covered by $O(K)$ lines of gradient $z$. So we may assume that $\tau(p_2)$ is the origin. 

Second, if $\delta_{\lambda}$ is the transformation that sends $(x,y)$ to $(x,\lambda y)$, then $(x,y)$ lies on the line $y=mx+c$ if and only if $\delta_{\lambda}(x,y)$ lies on the line $y=(\lambda m) x+ \lambda c$. So $\delta_{\lambda}(E)$ is supported on $O(K)$ lines of gradient $\lambda z$ and $O(K)$ lines through the origin. By taking $\lambda = 1/z$ we may assume therefore that $z=1$. 

So without loss of generality we may assume that $E$ is supported on $O(K)$ lines through the origin and $O(K)$ lines of gradient 1. In other words
$$\left|A \stackrel{E}{/}B\right|,\left|A \stackrel{E}{-}B\right|\ll K.$$

\subsection{Concluding the proof}

We want to apply sum-product (or strictly speaking, difference-ratio) estimates to the \textit{partial} difference and ratio sets that we have just constructed. To do this we need to work with \textit{complete} difference and ratio sets. The traditional machinery for making such a transition is the Balog-Szemer\'edi-Gowers theorem. We employ a variation, itself a straightforward modification of one due to Bourgain and Garaev \cite{bg}:  

\begin{lemma}\label{theorem:bg}
Let $A,B \subseteq \mathbb{F}_p$ and $E \subseteq A \times B$. Then there exists $A' \subseteq A$ with $|A'|\gg \frac{|E|}{|B|}$ and $\frac{|A'-A'|}{\left|A\stackrel{E}{-}B\right|^4},\frac{|(A'/A'|}{\left|A\stackrel{E}{/}B\right|^4}\ll \frac{|A|^4|B|^3}{|E|^5}$.
\end{lemma}

A proof of Lemma \ref{theorem:bg} is given in the appendix. To apply Lemma \ref{theorem:bg} we note that since $K\gg N^{1/2-\epsilon}$ we have
$$\frac{|E|}{|B|}\approx \frac{K^5}{N^{2+2 \epsilon}}\gg N^{1/2-7 \epsilon}$$
and 
$$\frac{|A|^4|B|^3\left|A\stackrel{E}{-}B\right|^4}{|E|^5}, \frac{|A|^4|B|^3\left|A\stackrel{E}{/}B\right|^4}{|E|^5} \ll \frac{N^{10+10\epsilon}}{K^{19}}  \ll N^{1/2+29 \epsilon}.$$

So there exists $A' \subseteq A$ with $|A'|\gtrsim N^{1/2-7 \epsilon}$ and $|A'-A'|,|A'/A'|\lesssim N^{1/2+29\epsilon}$. Applying Lemma \ref{theorem:rudnev} to $A'$ yields $N^{1/2}\lesssim N^{403 \epsilon}$ and so $\epsilon \geq \frac{1}{806}-o(1)$, which completes the proof of Theorem \ref{theorem:result}.

\section*{Appendix: Proof of Lemma \ref{theorem:bg}}

The proof of Lemma \ref{theorem:bg} follows the approach in \cite{bg} closely, making only a slight technical modification. The following lemma can be found in the book of Tao and Vu \cite{TV}:

\begin{lemma}
Let $E \subseteq A \times B$. Then for any $\epsilon >0$ there exists $A' \subseteq A$ with $|A'|\gg \frac{|E|}{|B|} $ such that for at least $(1-\epsilon)|A'|^2$ of the pairs $(a_1',a_2') \in A'\times A'$ we have 

$$\left|N(a_1')\cap N(a_2')\right|\geq \frac{\epsilon |E|^2}{2|A|^2|B|}$$

where $N(a)$ is the set of $b\in B$ for which $(a,b) \in E$.
\end{lemma} 

We apply it to obtain:

\begin{lemma}
Let $G(A,B,E)$ be a bipartite graph. Then there exists $A'\subseteq A$ with $|A'|\gg \frac{|E|}{|B|}$ such that every pair of elements from $A' \times A'$ is connected by $\Omega\left(\frac{|E|^5}{|A|^4|B|^3}\right)$ paths of length four in $E$.
\end{lemma}

\begin{proof}
Say that $(a_1,a_2)\in A \times A$ is \textbf{good} if
$$\left|N(a_1)\cap N(a_2)\right|\geq 0.05 \frac{|E|^2}{|A|^2|B|}.$$ 
By the previous lemma we can find $|A'|\gg \frac{|E|}{|B|}$ such that $0.9 |A'|^2$ of pairs from $A'$ are good.

Given $a_1'\in A'$ denote by $I_{a_1'}$ the set of elements $a_2' \in A'$ for which $(a_1',a_2')$ is good. Then we have

$$\sum_{a' \in A'}|I_{a'}|\geq 0.9 |A'|^2.$$

So by popularity pigeonholing there exists $|A''|\gg|A'|$ such that $|I_a|\geq 0.7|A'| $ for every $a \in A''$. So for any pair $a_1,a_2 \in A''$ we have $|I_{a_1}\cap I_{a_2}|\gg |A'|$ and so there are $\Omega(|A'|)$ elements $c \in A'$ for which
$$|N(a_1)\cap N(c)|,|N(a_2)\cap N(c)|\gg \frac{|E|^2}{|A|^2|B|}.$$ This means that there are $\Omega \left(\frac{|E|^5}{|A|^4|B|^3}\right)$ paths of length four in $E$ connecting $a_1$ and $a_2$, as required. 
\end{proof}

To prove Lemma \ref{theorem:bg} we note that for each $\alpha,\beta \in A'$ we have, 
$$\#\left\{(b_1,a,b_2)\in B \times A \times B:(\alpha,b_1),(a,b_1),(a_1,b_2),(\beta,b_2)\in E\right\}\gg \frac{|E|^5}{|A|^4|B|^3}.$$

Now it is clear that 
\begin{align*}
\alpha-\beta&= (\alpha-b_1)-(c-b_1)+(a-b_2)-(\beta-b_2)\\ 
\frac{\alpha}{\beta}&=\frac{\alpha}{b_1}\cdot \frac{b_1}{a} \cdot \frac{a}{b_2}\cdot \frac{b_2}{\beta}.
\end{align*}

So for all $\alpha,\beta \in A'$ we have

$$\#\left\{(s,t,u,v)\in A\stackrel{E}{-}B:s-t+u-v=\alpha-\beta\right\}\gg\frac{|E|^5}{|A|^4|B|^3}  $$ 
$$\#\left\{(s,t,u,v)\in A\stackrel{E}{\backslash}B:\frac{st}{uv}=\frac{\alpha}{\beta}\right\}\gg \frac{|E|^5}{|A|^4|B|^3}.$$

Summing over all elements of $A'-A'$ and $A'/A'$ respectively, we obtain

$$\frac{|E|^5}{|A|^4|B|^3}|A'-A'|\ll\left|A \stackrel{E}{-}B \right|^4 $$
$$\frac{|E|^5}{|A|^4|B|^3}|A'/A'|\ll\left|A \stackrel{E}{/}B\right|^4. $$

Rearranging gives the statment of Lemma \ref{theorem:bg}.

\bibliographystyle{plain}
\bibliography{Incidencesbibliography}

\end{document}